\documentclass[12pt]{amsart}
\usepackage{latexsym, amsmath, amscd, amssymb, amsthm} 
\usepackage[T2A]{fontenc}
 \usepackage[cp1251]{inputenc}
\usepackage[english]{babel}
\usepackage{color}
\newtheorem{theorem}{Theorem}
\newtheorem{thm}{Theorem}
\newtheorem{defn}{Definition}[section]

\catcode`,\active

\catcode`\,12

\begin{document}

\title{Derivations and identities for Chebyshev polynomials of  the first and second kinds}           
{}                 

\author{Leonid Bedratyuk}             
\address{Khmelnytskyi National University, 11, Instytuts'ka st., 29016, Khmelnytskyi, Ukraine \\ \\ 
 leonid.uk@gmail.com}                  
        
\author{Nataliia Luno}            
\address{Vasyl' Stus Donetsk National University, 600-richchia, 21st., 21021, Vinnytsia, Ukraine \\
nlunio@ukr.net}

\keywords{Chebyshev polynomials, derivations, generalised hypergeometric function.}               

\begin{abstract}
In this note we follow the general approach, proposed by the first author, which is derived from the invariant theory field and provides a way of obtaining of the polynomial identities for any arbitrary polynomial family.  
We introduce the notion of Chebyshev derivations of the first and second kinds, which is  based on the polynomial algebra, and corresponding specific differential operators. We derive the elements of their kernels and prove that any element of the kernel of the derivations  defines a polynomial identity satisfied by the Chebyshev polynomials of the first and second kinds.  Combining elementary methods and combinatorial techniques, we obtain  several new polynomial identities involving the Chebyshev polynomials of the both kinds and a special case of the Jacobi polynomials.
Using the properties of the generalised hypergeometric function, we specify  the Chebyshev polynomials of the first and second kinds  via  the generalised hypergeometric function and, as a consequence, derive the corresponding identities involving the generalised hypergeometric function and the Chebyshev polynomials of the first and second kinds.
\end{abstract}

\maketitle
\section{Introduction}

The Chebyshev polynomials of the first kind    $T_n(x)$  and the Chebyshev polynomials of the second kind  $U_n(x)$   are defined by the following ordinary generating functions
\begin{gather*}
\mathcal{G}(T_n(x),t)=\frac{1-xt }{1-2xt+t^2}=\sum_{n=0}^{\infty} T_n(x)  t^n,\\
 \mathcal{G}(U_n(x),t)=\frac{1}{1-2xt+t^2}=\sum_{n=0}^{\infty}U_n(x)  t^n. 
\end{gather*}

The derivatives of the polynomials can be expressed in terms of the polynomials as follows:

 \begin{align*}
 &\frac{d}{dx} T_n(x)=n \left( \sum_{k=1}^{n-1} \left(1-(-1)^k \right) T_{n-k}(x) +\frac{1-(-1)^n}{2}T_0(x) \right), \label{der1}\\
 &\frac{d}{dx} U_n(x)=\sum_{k=1}^{n} \left( 1-(-1)^{n-k} \right) (n-k+1)U_{n-k}(x)=\sum_{k=0}^{\left[\frac{n}{2} \right]}  (n-2k)U_{n-2k-1}(x),
 \end{align*}
  see  \cite{BI3}.

    We are interested in finding  polynomial identities satisfied by the  polynomials, i.e.,  identities  of the form
$$
P(T_0(x),T_1(x),\ldots, T_n(x))={\rm 0,} \,\,\text{  or }  P(U_0(x),U_1(x),\ldots, U_n(x))={\rm 0},
$$
where  $P(x_0,x_1,\ldots,x_n)$ is a polynomial of  $n+1$ variables.

We are going to use the approach proposed by  the first author in \cite{BI1}. We provide a method for finding such identities that is based on the simple observation:

if 
\begin{equation*}\label{eqn-zero}
\frac{d}{dx} P(T_0(x),T_1(x),\ldots, T_n(x))=0,
\end{equation*}
take place, then  
$
P(T_0(x),T_1(x),\ldots, T_n(x))
$
is constant.  In other words,
$P(T_0(x),T_1(x),\ldots, T_n(x))$ gives an identity of the Chebyshev polynomials of the first kind.

Now, we rewrite the derivative as:
\begin{gather*}
\frac{d}{dx} P(T_0(x),T_1(x),\ldots, T_n(x))=\\=\frac{\partial }{\partial x_0}P(x_0,x_1,\ldots,x_n)\Big |_{\{x_i=T_i(x)\}} \frac{d}{dx} T_0(x)+\cdots +\frac{\partial }{\partial x_n}P(x_0,x_1,\ldots,x_n)\Big |_{\{x_i=T_i(x)\}} \frac{d}{dx} T_n(x).
\end{gather*}
Then,
 
 \begin{gather*}
\frac{d}{dx} P(T_0(x),T_1(x),\ldots, T_n(x))=\\=\left( \frac{\partial }{\partial x_0}P(x_0,x_1,\ldots,x_n) \mathcal{D}_{\mathcal{T}}(x_0) +\cdots +\frac{\partial }{\partial x_n}T(x_0,x_1,\ldots,x_n)\mathcal{D}_{\mathcal{T}}(x_n)\right) \Big |_{\{x_i=T_i(x)\}} =\\=
\mathcal{D}_{\mathcal{T}}(P(x_0,x_1,\ldots,x_n))\Big |_{\{x_i=T_i(x)\}},
\end{gather*}
 where the differential operator  $\mathcal{D}_{\mathcal{T}}$ is defined by 
 
 $$
 \mathcal{D}_{\mathcal{T}}=n \left( \sum_{k=1}^{n-1} \left(1-(-1)^k \right) x_{n-k} +\frac{1-(-1)^n}{2}x_0 \right)
.
$$
 It is clear that if $\mathcal{D}_{\mathcal{T}}(P(x_0,x_1,\ldots,x_n))=0$ then $\dfrac{d}{dx} P(T_0(x),T_1(x),\ldots, T_n(x))$  is a constant.
 Thus, any non-trivial polynomial  $P(x_0,x_1,\ldots,x_n)$, which belongs to the kernel of   $\mathcal{D}_{\mathcal{T}},$ defines a polynomial identity relating the Chebyshev polynomials of the first kind. 
 
Similar construction for the Chebyshev polynomials of the second kind takes place. Let us introduce the differential operator 
 $$
  \mathcal{D}_{\mathcal{U}}(x_n)=\sum_{k=1}^{n-1} \left( 1+(-1)^{n-k+1} \right) (k+1)x_k(x).
 $$
Then   the condition $$ \mathcal{D}_{\mathcal{U}}(P(x_0,x_1,\ldots,x_n))=0,$$ defines a  polynomial  identity  
for  the Chebyshev polynomials of the second kind.

In the paper, using this approach, we  introduce the notion of Chebyshev derivations of the first and second kinds, derive the elements of their kernels and find the corresponding identities for the Chebyshev polynomials of both kinds:

\begin{multline*}
T_n(x)+\sum_{k=1}^{n} \frac {(-2)^kn}{k!}\left(  \sum_{i=0}^{\left[\frac{n-k}{2} \right]} (n-i-1)^{ \underline{k-1}} {k{+}i{-}1 \choose k{-}1}T_{n-k-2i}(x){T_{1}(x)}^k   -\right. \\
\left.- (1+(-1)^{n-k}) \frac{1}{4} {\left(\frac{n+k}{2}-1 \right)}^{ \underline{k-1}}{{ {\frac{n+k}{2}-1} \choose k-1} }{T_{1}(x)}^k   \right)=\cos \frac{\pi n}{2}.
\end{multline*}
	\begin{multline*}
U_n(x)+\sum_{k=1}^{n}\frac{{(-1)^{k}}}{k!} \sum_{i=0}^{\left[\frac{n-k}{2}\right]}  ( n-2i-k+1)(n-i)^{ \underline{k-1}}{k+i-1 \choose k-1} U_{n-k-2i}(x) {U_1(x)}^k =\\
=\cos \frac{\pi n}{2}.
\end{multline*}

\section{The Chebyshev derivations}\label{secondSec}

 Let   $\mathbb{C}[x_0,x_1,x_2,\ldots,x_n]$ be the polynomial algebra  in  $n+1$ variables $x_0,x_1,x_2,\ldots,x_n$ over $\mathbb{C}.$ Recall  that a  {\it derivation} of the polynomial algebra $\mathbb{C}[x_0,x_1,x_2,\ldots,x_n]$ is a linear  map  $D$ satisfying the Leibniz rule: 
$$
D(f \, g)=D(f) g+f D(g), \text{  for all }  f, g \in \mathbb{C}[x_0,x_1,x_2,\ldots,x_n].
$$
By  using the quotient rule of derivations we can extend a derivation to the field of
fractions $\mathbb{C}(x_0,x_1,x_2,\ldots,x_n)$.

A derivation $D$  is called {\it locally nilpotent} if for every $f \in \mathbb{C}[x_0,x_1,x_2,\ldots,x_n]$ there exists  an $n \in \mathbb{N}$ such that $D^{n+1}(f)=0$  but  $D^{n}(f)  \neq 0.$
Any derivation   $D$ is completely determined by the elements $D(x_i).$ A  derivation   $D$  is called  \textit{linear} if  $D(x_i)$ is a linear form. A  linear locally nilpotent derivation is called a \textit{Weitzenb\"ock derivation}. 
A derivation    $D$ 
is called a triangular if  $D(x_i) \in \mathbb{C}[x_0,\ldots,x_{i-1}] \forall i \leq n.$ Any triangular derivation is locally nilpotent.

The subalgebra 
$$
\ker D:=\left \{ f \in \mathbb{C}[x_0,x_1,x_2,\ldots,x_n] \mid  D(f)=0 \right \},
$$
is called the {\it kernel} of derivation $D.$

It is known \cite{BI7}, that if for an arbitrary locally nilpotent derivation $D$ there exists  polynomials $h$ such that  $D(h) \neq 0$ but $D^2(h)=0,$ then
$$
\ker D=\mathbb{C}[\sigma(x_0),\sigma(x_1),\ldots,\sigma(x_n)][D(h)^{-1}] \cap \mathbb{C}[x_0,x_1,\ldots,x_n],
$$
and the element
$$ \sigma_D(x_n)=\sum_{k=0}^{\infty} D^k(x_n) \frac{\lambda^k}{k!}, \: \small{\text{where} \:\lambda=-\frac{h}{D(h)}, D(\lambda)=-1}, \eqno{(3)}
$$ 
 belongs to the kernel of the derivation $\ker D$.

\begin{defn}
Elements $x_0^{n-1} \sigma(x_n)$, which belong to the kernel $\ker \mathcal{D}_{\mathcal{T}},$ we call   \textit{the Cayley elements of the Chebyshev first kind derivation} of the locally nilpotent derivation $\mathcal{D}_{\mathcal{T}}.$
\end{defn}

This motivates the following 
\begin{defn}
Derivations of  $\mathbb{C}[x_0,x_1,x_2,\ldots,x_n]$  defined by 
 \begin{align*}
&D_\mathcal{T}(x_0)=0, D_\mathcal{T}(x_n)=n \left( \sum_{k=1}^{n-1} \left(1-(-1)^k \right) x_{n-k} +\frac{1-(-1)^n}{2}x_0 \right),\\
&D_\mathcal{U}(x_0)=0, D_\mathcal{U}(x_n)=\sum_{k=1}^{n-1} \left( 1+(-1)^{n-k+1} \right) (k+1)x_k,
\end{align*}
are called  {\bf the Chebyshev first kind derivation} and {\bf the Chebyshev second kind derivation}  respectively.
\end{defn}
We  have
$$
\begin{array}{ll}
D_\mathcal{T}(x_0)=0,& D_\mathcal{U}(x_0)=0,\\
D_\mathcal{T}(x_1)=x_0,&D_\mathcal{U}(x_1)=2x_0,\\
 D_\mathcal{T}(x_2)=4x_1,& D_\mathcal{U}(x_2)=4x_1,\\
 D_\mathcal{T}(x_3)=6\,x_{{2}}+3\,x_{{0}},& D_\mathcal{U}(x_3)=2(3\,x_{{2}}+\,x_{{0}}),\\
 D_\mathcal{T}(x_4)=8\,x_{{3}}+8\,x_{{1}},& D_\mathcal{U}(x_4)=2(4\,x_{{3}}+2\,x_{{1}}),\\
 D_\mathcal{T}(x_5)=10\,x_{{4}}+10\,x_{{2}}+5\,x_{{0}},& D_\mathcal{U}(x_5)=2(5\,x_{{4}}+3\,x_{{2}}+\,x_{{0}}),\\
 D_\mathcal{T}(x_6)=12\,x_{{5}}+12\,x_{{3}}+12\,x_{{1}}, & D_\mathcal{U}(x_6)=2(6\,x_{{5}}+4\,x_{{3}}+2x_{{1}}),\\
 D_\mathcal{T}(x_7)=14\,x_{{6}}+14\,x_{{4}}+14\,x_{{2}}+7\,x_{{0}},& D_\mathcal{U}(x_7)=2(7\,x_{{6}}+5\,x_{{4}}+3\,x_{{2}}+\,x_{{0}}),\\
 D_\mathcal{T}(x_8)=16\,x_{{7}}+16\,x_{{5}}+16\,x_{{3}}+16x_{{1}}, & D_\mathcal{U}(x_8)=2(8\,x_{{7}}+6\,x_{{5}}+4x_{{3}}+2\,x_{{1}}).
\end{array}
$$

\section{The kernel of the Chebyshev  derivation  of the first kind}

As it was proved in \cite{BI2}, the $k$-th derivative of Chebyshev polinomial of the first kind has the form

\begin{multline*}
D^k_\mathcal{T}(x_n)={2^k} \sum_{i=0}^{\left[\frac{n-k}{2} \right]}  (n-i-1)^{ \underline{k-1}}  {k+i-1 \choose k-1} x_{n-k-2i}-\\ -[[n-k\ \textbf{even} ]]{2^{k-1}} {\left(\frac{n+k}{2}-1 \right)}^{ \underline{k-1}} { {\frac{n+k}{2}-1} \choose k-1}x_0,
\end{multline*}

where notations of falling factorials $x^{ \underline{n}}=x(x-1)\cdot \cdots \cdot (x-n+1)$ and Iverson's symbol $[P]$ which is 1 if $P$ is true and 0 otherwise is used.

After eliminating the Iverson's symbol 
and falling factorials 
 we get 
\begin{multline*}
D^k_\mathcal{T}(x_n)={2^k} \sum_{i=0}^{\left[\frac{n-k}{2} \right]} (n-i-1)^{ \underline{k-1}}   {k+i-1 \choose k-1} x_{n-k-2i}-\\ -\frac{1}{2}  (1+(-1)^{n-k}) 2^{k-1}n {\left(\frac{n+k}{2}-1 \right)}^{ \underline{k-1}}  {\left(\frac{n-k}{2}\right)!} {{ {\frac{n+k}{2}-1} \choose k-1} } x_0.
\end{multline*}

Since  $D_\mathcal{T}\left(-\displaystyle  \frac{x_1}{x_0} \right)=-1, $  we put $\lambda=- \displaystyle  \frac{x_1}{x_0}.$ Now we may find the Diximier map:
\begin{multline*}
\sigma(x_n)=\sum_{k=0}^{n}D^k_\mathcal{T}(x_n) \frac{\lambda^k}{k!}=x_n +\left[ \sum_{k=1}^{n} \frac{\lambda^k}{k!} {2^k}  \sum_{i=0}^{\left[ \frac{n-k}{2} \right]}  (n-i-1)^{ \underline{k-1}}   {k+i-1 \choose k-1} x_{n-k-2i} \right.- \\ - \left. \frac{1}{2}  (1+(-1)^{n-k}) 2^{k-1}n {\left(\frac{n+k}{2}-1 \right)}^{ \underline{k-1}} {{ {\frac{n+k}{2}-1} \choose k-1} } x_0 \right]. 
\end{multline*}

Replacing  $\lambda$ by $-\displaystyle  \frac{x_1}{x_0} $, we obtain, after   simplifying:
\begin{multline*}
x_0^{n-1} \sigma(x_n)=  x_n x_0^{n-1}+\sum_{k=1}^{n} \frac{(-2)^kn}{k!} \left[ \sum_{i=0}^{\left[\frac{n-k}{2} \right]}(n-i-1)^{ \underline{k-1}}  {k+i-1 \choose k-1} x_{n-k-2i}x_1^k x_0^{n-1-k}
\right. -\\
- \left. n \frac {(1+(-1)^{n-k})}{4}{\left(\frac{n+k}{2}-1 \right)}^{ \underline{k-1}} {{ {\frac{n+k}{2}-1} \choose k-1} } x_1^{k}x_0^{n-k}  \right].
\end{multline*}

Thus, we prove the following
\begin{thm}\label{t1} The kernel of the derivation $D_\mathcal{T}$ generating by the Cayley elements of the Chebyshev first kind derivation has the following form:
\begin{multline*}
C_{\mathcal{T}}(x_0,x_1, \ldots, x_n)={x_n}x_0^{n-1}+\sum_{k=1}^{n} \frac {(-2)^kn}{k!}\left(  \sum_{i=0}^{\left[\frac{n-k}{2} \right]}  (n-i-1)^{ \underline{k-1}} {k{+}i{-}1 \choose k{-}1}x_{n-k-2i}x_{1}^k x_{0}^{n-1-k} -\right. \\
\left.-(1+(-1)^{n-k}) \frac{1}{4} {\left(\frac{n+k}{2}-1 \right)}^{ \underline{k-1}}{{ {\frac{n+k}{2}-1} \choose k-1} } x_{1}^k x_{0}^{n-k}  \right).
\end{multline*}
 
\end{thm}

The first few Cayley elements are:
\begin{align*}
&C_\mathcal{T}(x_0,x_1, x_2)=-2{x_{{1}}}^{2}+x_{{2}}x_{{0}},\\
&C_\mathcal{T}(x_0,x_1, x_2,x_3)=8\,{x_{{1}}}^{3}-6\,{x_{{1}}x_{{2}}x_{{0}}}-3\,{x_{{0}}}^{2}x_{{1}}+x_{{3}}
{x_{{0}}}^{2},\\
&C_\mathcal{T}(x_0,x_1, x_2,x_3,x_4)=-24\,{x_{{1}}}^{4}+24\,{x_{{1}}}^{2}x_{{2}}x_{{0}}+8{x_{{0}}}^{2}{x_{{1}}
}^{2}-8\,x_{{1}}x_{{3}}{x_{{0}}}^{2}+x_{{4}}{x_{{0}}}^{3},\\
&C_\mathcal{T}(x_0,x_1, x_2,x_3,x_4,x_5)=64\,{x_{{1}}}^{5}-80\,{x_{{1}}}^{3}x_{{2}}x_{{0}}+
40\,{x_{{1}}}^{2}x_{{3}}{x_{{0}}}^{2}-10\,x_{{1}}x_{{4}}{x_{{0
}}}^{3}-10\,{x_{{0}}}^{3}x_{{1}}x_{{2}}-\\&-5\,{x_{{0}}}^{4}x_{{1}}+x_{{5}}{x_
{{0}}}^{4}.
\end{align*} 

Going back to the Chebyshev polynomials  we get  the following statement.

\begin{theorem}  

For any arbitrary $n \in \mathbb{N},$ the following identities hold
$$
\begin{array}{ll}
(i)& \displaystyle 
T_n(x)+n \sum_{k=1}^{n}  (-2 T_{1}(x))^k \left(  \sum_{i=0}^{\left[\frac{n-k}{2} \right]} \frac{1}{n-i} \binom{n-i}{k} {k{+}i{-}1 \choose k{-}1}T_{n-k-2i}(x) \right)   = \\=& \displaystyle 
n \sum_{k=1}^{n} \left((1+(-1)^{n-k}) \frac{1}{4} {\left(\frac{n+k}{2}-1 \right)}^{ \underline{k-1}}{{ {\frac{n+k}{2}-1} \choose k-1} }{T_{1}(x)}^k   \right)+\cos \frac{\pi n}{2},\\
(ii)& \displaystyle 
\sum_{k=0}^n \left( \sum_{i=0}^{\left[\frac{n-k}{2} \right]} {(-2)^{n-(k+2i)}}  \frac{1}{n-i} \binom{n-i}{k+i} {n-k-i-1 \choose i} T_1(x)^{n-(k+2i)} \right) T_k(x)=\\&=\displaystyle 
\sum_{k=0}^n (1+(-1)^{n-k})  \frac{k (-2)^{k} }{(n+k)^2} \binom{\frac{n+k}{2}}{k}^2 T_1(x)^k +\frac{1}{n}\cos \frac{\pi n}{2},\\
(iii) & \displaystyle 
{P^{\left(-\frac{1}{2},-\frac{1}{2}\right)}_{n}(x)}+\sum_{k=1}^{n} \frac {(-x)^k (n+k-1)!}{ k! 2^k(n-1)!}   {P^{\left(k-\frac{1}{2},k-\frac{1}{2}\right)}_{n-k}(x)}  =  {\left(\frac{1}{2}
\right)}_{ n} \frac{1}{n!}\cos \frac{\pi n}{2},
\end{array}
$$
$$
\text{ where} \:{P^{\left(\alpha,\beta \right)}_{n}(x)}\text {is} \:\textit{the Jacobi} \:\textit{polynomial}\: such \: that  \:deg P^{\left(\alpha,\beta \right)}_{n}(x)=n.
$$
\end{theorem}

\begin{proof}
$(i)$ By Theorem~\ref{t1} we have \begin{multline*}
T_n(x)+\sum_{k=1}^{n} \frac {(-2)^kn}{k!}\left(  \sum_{i=0}^{\left[\frac{n-k}{2} \right]} (n-i-1)^{ \underline{k-1}} {k{+}i{-}1 \choose k{-}1}T_{n-k-2i}(x){T_{1}(x)}^k   -\right. \\
\left.-  (1+(-1)^{n-k}) \frac{1}{4} {\left(\frac{n+k}{2}-1 \right)}^{ \underline{k-1}}{{ {\frac{n+k}{2}-1} \choose k-1} }{T_{1}(x)}^k   \right)=t,
\end{multline*}
for some constant $t$.

As far, as the left side of the latter identity does not depend on $x,$ to find the unknown constant $t$ we put $x=0,$ then

\begin{multline*}
T_n(0)+\sum_{k=1}^{n} \frac {(-2)^kn}{k!}\left(  \sum_{i=0}^{\left[\frac{n-k}{2} \right]} (n-i-1)^{ \underline{k-1}} {k{+}i{-}1 \choose k{-}1}T_{n-k-2i}(0){T_{1}(0)}^k   -\right. \\
\left.- (1+(-1)^{n-k}) \frac{1}{4} {\left(\frac{n+k}{2}-1 \right)}^{ \underline{k-1}}{{ {\frac{n+k}{2}-1} \choose k-1} }{T_{1}(0)}^k   \right)=t.
\end{multline*}
Since $T_1(x)=x$, then $T_1(0)=0$ and we have $t=T_n(0).$ 
Using the explicit form of the Chebyshev first kind polynomials 
$$
T_n(x)=\sum_{k=0}^{\frac{n}{2}} \binom{n}{2k} (x^2-1)^k x^{n-2k},
$$
we find that  $T_n(0)=0$  for the odd $n,$ and $T_n(0)=(-1)^{\frac{n}{2}}$ for the even $n.$  Combining the both cases, we obtain
$$
t=T_n(0)=\cos \frac{\pi n}{2}.
$$

$(ii)$  Taking into account properties of non-integers binomial coefficients and their connection with falling factorials, we have
\begin{multline*}
\binom{\frac{n+k}{2}-1}{k-1}=\frac{\Gamma \left(\frac{n+k}{2}\right)}{\Gamma \left(\frac{n-k}{2}+1\right)(k-1)!}=\frac{\Gamma \left(\frac{n+k}{2}\right)\left(\frac{n+k}{2}\right)}{\Gamma \left(\frac{n-k}{2}+1\right)(k-1)!k}\frac{k}{\frac{n+k}{2}}=\\
=\frac{\Gamma \left(\frac{n+k}{2}+1\right)}{\Gamma \left(\frac{n-k}{2}+1\right)k!}\frac{k}{\frac{n+k}{2}}=\frac{2k}{n+k}\binom{\frac{n+k}{2}}{k},
\end{multline*}
$$
{\left(\frac{n+k}{2}-1 \right)}^{ \underline{k-1}}=(k-1)!\binom{\frac{n+k}{2}-1}{k-1}=
\frac{2k!}{n+k}\binom{\frac{n+k}{2}}{k}.
$$

Note that after changing the order of summation by the general rule \cite{BI8}
\begin{gather*}
\sum_{k=0}^n \sum_{i=0}^{\left[\frac{n-k}{2} \right]} a_{k,i} x_{n-k-2i} =\sum_{k=0}^n \left( \sum_{i=0}^{\left[\frac{n-k}{2} \right]} a_{n-(k+2i),i} \right) x_k,
\end{gather*} 
and using the symmetry of binomial coefficients, the identity $(i)$ can be rewritten as  
\begin{gather*}\label{1}
\sum_{k=0}^n \left( \sum_{i=0}^{\left[\frac{n-k}{2} \right]} {(-2)^{n-(k+2i)}}  \frac{1}{n-i} \binom{n-i}{k+i} {n-k-i-1 \choose i} T_1(x)^{n-(k+2i)} \right) T_k(x)-\\-
\sum_{k=0}^n (1+(-1)^{n-k})  \frac{k (-2)^{k} }{(n+k)^2} \binom{\frac{n+k}{2}}{k}^2 T_1(x)^k =\frac{1}{n}\cos \frac{\pi n}{2}.
\end{gather*}

$(iii)$  Using  the formula for the $k$-th Jacobi polynomial derivative \cite{BI5}

$$
{\displaystyle {\frac {d^{k}}{dz^{k}}}P_{n}^{(\alpha ,\beta )}(z)={\frac {\Gamma (\alpha +\beta +n+1+k)}{2^{k}\Gamma (\alpha +\beta +n+1)}}P_{n-k}^{(\alpha +k,\beta +k)}(z)},
$$

and the fact that the Chebyshev polynomials of the first kind are a special case of the Jacobi polynomials \cite{BI4} $P_n^{(\alpha,\beta)}$ with $\alpha=\beta=-1/2$,
$$
T_n(x)=\frac{P^{\left(-\frac{1}{2},-\frac{1}{2}\right)}_n(x)}{P^{\left(-\frac{1}{2},-\frac{1}{2}\right)}_n(1)}=n! \frac{P^{\left(-\frac{1}{2},-\frac{1}{2}\right)}_n(x)}{(\frac{1}{2})_n},
$$

we obtain the $k$-th Chebyshev polynomial derivative in the terms of the Jacobi  polynomials

$$
\frac{d^k}{dx^k} T_n(x)=\frac{n!}{(\frac{1}{2})_n} \frac{(n+k-1)!}{2^k(n-1)!}
{P^{\left(k-\frac{1}{2},k-\frac{1}{2}\right)}_{n-k}(x)}.
$$

Now, with $T_1(x)=x,$ we can write the identity for the special case of the Jacobi polynomials
\begin{gather*}
 {P^{\left(-\frac{1}{2},-\frac{1}{2}\right)}_{n}(x)}+\sum_{k=1}^{n} \frac {(-x)^k (n+k-1)!}{ k! 2^k(n-1)!}   {P^{\left(k-\frac{1}{2},k-\frac{1}{2}\right)}_{n-k}(x)}  =  {\left(\frac{1}{2}
\right)}_{ n} \frac{1}{n!}\cos \frac{\pi n}{2}.
\end{gather*}
\end{proof}

After rewriting the $(ii)$ identity in the terms of hypergeometric function, we obtain the following

\begin{thm}
\begin{gather*}
\sum_{k=0}^n \binom{n}{k}(-2x)^{n-k}\,{}_4 F_3
\begin{bmatrix}
\begin{matrix}
&\dfrac{{-}n{+}k}{2},\dfrac{-n+k}{2}{+}1,\dfrac{{-}n{+}k}{2}{+}\dfrac 12,\dfrac{-n{+}k}{2}{+}\dfrac 12\:\\
&{-}n{+}1,k{+}1,{-}n{+}1{+}k
\end{matrix} \:
\bigg{\vert \frac{4}{{x}^{2}}}
\end{bmatrix}
 T_k(x)=\\=
n \sum_{k=0}^n (1+(-1)^{n-k})  \frac{k (-2)^{k} }{(n+k)^2} \binom{\frac{n+k}{2}}{k}^2 x^k +\cos \frac{\pi n}{2}.
\end{gather*}
\end{thm}
\begin{proof}

Denoting  the expression under the summation   sign  by $a_i$  
$$
a_i= {(-2)^{n-(k+2i)}}  \frac{1}{n-i} \binom{n-i}{k+i} {n-k-i-1 \choose i} x^{n-(k+2i)},
$$
and finding the ratio of $a_{i+1}$ and $a_i,$ after simplifying, we  obtain
\begin{gather*}
\frac{a_{i+1}}{a_{i}}={\frac { \left( k+2\,i+1-n \right) ^{2} \left( -n+k+2\,i \right) 
 \left( -n+k+2\,i+2 \right) }{ 4 \, \left( -n+i+1 \right)  \left( k+i+1
 \right)  \left( -n+k+i+1 \right)  \left( i+1 \right) {x}^{2}}},
\end{gather*}
with
$$
a_0=\dfrac {{n\choose k}}{n}(-2x)^{n-k}.
$$

Using the corresponding property of the generalized hypergeometric series \cite{BI6}, we have
\begin{gather*}
\sum_{i=0}^{\infty} {(-2)^{n-(k+2i)}}  \frac{1}{n-i} \binom{n-i}{k+i} {n-k-i-1 \choose i} x^{n-(k+2i)}=\\=
 \dfrac {{n\choose k}}{n}(-2x)^{n-k}
 {}_4 F_3
\begin{bmatrix}
\begin{matrix}
&\dfrac{{-}n{+}k}{2},\dfrac{-n+k}{2}{+}1,\dfrac{{-}n{+}k}{2}{+}\dfrac 12,\dfrac{-n{+}k}{2}{+}\dfrac 12\:\\
&{-}n{+}1,k{+}1,{-}n{+}1{+}k
\end{matrix} \:
\bigg{\vert \frac{4}{{x}^{2}}}
\end{bmatrix}
 =\\= \sum_{i=0}^{\left[\frac{n-k}{2} \right]} {(-2)^{n-(k+2i)}}  \frac{1}{n-i} \binom{n-i}{k+i} {n-k-i-1 \choose i} x^{n-(k+2i)}.
\end{gather*}
\end{proof}

\section{The kernel of the Chebyshev  derivation  of the second kind}

The $k$-th derivative of Chebyshev polinomial of the second kind has the form (see \cite{BI2})

\begin{gather*}
D^k_\mathcal{U}(x_n)={2^k} \sum_{i=0}^{\left[\frac{n-k}{2} \right]}  (n-i)^{ \underline{k-1}}  {k+i-1 \choose k-1} \left( n-k-2i+1 \right) x_{n-k-2i}.
\end{gather*}

Since  $D_\mathcal{U}\left(-\displaystyle  \frac{x_1}{2x_0} \right)=-1, $  we put $\lambda=- \displaystyle  \frac{x_1}{2x_0}$ and find the Diximier map:
\begin{multline*}
\sigma(x_n)=\sum_{k=0}^{n}D^k_\mathcal{U}(x_n) \frac{\lambda^k}{k!}=x_n +\\
+ \sum_{k=1}^{n} \frac{\lambda^k}{k!} {2^k}  \sum_{i=0}^{\left[ \frac{n-k}{2} \right]}  (n-i)^{ \underline{k-1}}   {k+i-1 \choose k-1}  \left( n-k-2i+1 \right) x_{n-k-2i}. 
\end{multline*}

Replacing  $\lambda$ by $-\displaystyle  \frac{x_1}{2x_0} $, we obtain, after   simplifying:
\begin{multline*}
x_0^{n-1} \sigma(x_n)=  x_n x_0^{n-1}+\\
+\sum_{k=1}^{n} \frac{(-1)^k}{k!} \sum_{i=0}^{\left[\frac{n-k}{2} \right]}(n-i-1)^{ \underline{k-1}}  {k+i-1 \choose k-1} \left( n-k-2i+1 \right) x_{n-k-2i}x_1^k x_0^{n-1-k}.
\end{multline*}

Thus, we prove the following
\begin{theorem}\label{t4} The kernel of the derivation $D_\mathcal{T}$ generating by the Cayley elements of the Chebyshev second kind derivation has the following form:
\begin{multline*}
C_{\mathcal{U}}(x_0,x_1, \ldots, x_n)={x_n}x_0^{n-1}+\\
+\sum_{k=1}^{n} \frac{(-1)^k}{k!} \sum_{i=0}^{\left[\frac{n-k}{2} \right]}(n-i-1)^{ \underline{k-1}}  {k{+}i{-}1 \choose k-1} \left( n{-}k{-}2i{+}1 \right) x_{n-k-2i}x_1^k x_0^{n-1-k}.
\end{multline*}
 
\end{theorem}

The first few Cayley elements are:
\begin{align*}
&C_\mathcal{U}(x_0,x_1, x_2)=-{x_{{1}}}^{2}+x_{{2}}x_{{0}},\\
&C_\mathcal{U}(x_0,x_1, x_2,x_3)=2\,{x_{{1}}}^{3}-3\,{x_{{1}}x_{{2}}x_{{0}}}-\,{x_{{0}}}^{2}x_{{1}}+x_{{3}}
{x_{{0}}}^{2},\\
&C_\mathcal{U}(x_0,x_1, x_2,x_3,x_4)=-3\,{x_{{1}}}^{4}+6\,{x_{{1}}}^{2}x_{{2}}x_{{0}}+{x_{{0}}}^{2}{x_{{1}}
}^{2}-4\,x_{{1}}x_{{3}}{x_{{0}}}^{2}+x_{{4}}{x_{{0}}}^{3},\\
&C_\mathcal{U}(x_0,x_1, x_2,x_3,x_4,x_5)=4\,{x_{{1}}}^{5}-10\,{x_{{1}}}^{3}x_{{2}}x_{{0}}+
10\,{x_{{1}}}^{2}x_{{3}}{x_{{0}}}^{2}+2\,{x_{{0}}}^{2}x_{{1}}{x_{{1}}}^{3}-\\&-5\,x_{{1}}x_{{4}}{x_{{0
}}}^{3}-3\,{x_{{0}}}^{3}x_{{1}}x_{{2}}-\,{x_{{0}}}^{4}x_{{1}}+\,{x_
{{0}}}^{4}x_{{5}}.
\end{align*} 
Going back to the Chebyshev polynomials  we get  the following statement.

\begin{theorem}  

For any arbitrary $n \in \mathbb{N},$ the following identities hold

$$
\begin{array}{ll}
(i)& \displaystyle 
U_n(x)+ \sum_{k=1}^{n} \frac{(- U_{1}(x))^k }{k} \left(  \sum_{i=0}^{\left[\frac{n-k}{2} \right]}  \binom{n-i}{k-1} {k{+}i{-}1 \choose k{-}1}  \left( n{-}k{-}2i{+}1 \right)U_{n-k-2i}(x) \right) 
=\\&=\displaystyle \cos \frac{\pi n}{2},\\
(ii)& \displaystyle 
\sum_{k=0}^n \left( \sum_{i=0}^{\left[\frac{n-k}{2} \right]} {(-1)^{n-(k+2i)}}  \frac{k+1}{i+k+1} \binom{n-i}{k+i} {n-k-i-1 \choose i} U_1(x)^{n-(k+2i)} \right) U_k(x)
=\\&=\displaystyle \cos \frac{\pi n}{2},\\
(iii)& \displaystyle 
(n+1){P^{\left(\frac{1}{2},\frac{1}{2}\right)}_{n}(x)}+\sum_{k=1}^{n} \frac {(-x)^k (n+k+1)!}{ k! 2^k  n!}   {P^{\left(k+\frac{1}{2},k+\frac{1}{2}\right)}_{n-k}(x)}  =  
 {n+\frac{1}{2} \choose n}  \cos \frac{\pi n}{2},
\end{array}
$$
$$
\text{ where} \:{P^{\left(\alpha,\beta \right)}_{n}(x)}\text {is} \:\textit{the Jacobi} \:\textit{polynomial}\: such \: that  \:deg P^{\left(\alpha,\beta \right)}_{n}(x)=n.
$$
\end{theorem}

\begin{proof}
$(i)$ By Theorem~\ref{t4} we have \begin{gather*}
U_n(x)+\sum_{k=1}^{n} \frac {(-1)^k}{k!} \sum_{i=0}^{\left[\frac{n-k}{2} \right]} (n-i)^{ \underline{k-1}} {k{+}i{-}1 \choose k{-}1}\left( n{-}k{-}2i{+}1 \right)U_{n-k-2i}(x){U_{1}(x)}^k =t,
\end{gather*}
for some constant $t$.

Using the same method as in the case of the Chebyshev first kind derivation, we find the unknown constant $t:$ 
$$
t=U_n(0)=\cos \frac{\pi n}{2}.
$$

$(ii)$  
With the property of binomial coefficients 
$$
 {n{-}i \choose k{-}1}=\frac{k}{ n{-}k{-}i{+}1}  {n{-}i \choose k},
$$
we obtain 
\begin{gather*}U_n(x)+ \sum_{k=1}^{n} \frac{{(- 1)}^k }{k} \left(  \sum_{i=0}^{\left[\frac{n-k}{2} \right]}  \binom{n{-}i}{k{-}1} {k{+}i{-}1 \choose k{-}1}  \left( n{-}k{-}2i{+}1 \right)U_{n-k-2i}(x) {(- U_{1}(x))}^k \right) =\\
=U_n(x)
+\sum_{k=1}^{n}  \left(  \sum_{i=0}^{\left[\frac{n-k}{2} \right]}\frac{{(- 1)}^k \left( n{-}k{-}2i{+}1 \right)}{ n{-}k{-}i{+}1 } \binom{n{-}i}{k} {k{+}i{-}1 \choose k{-}1}  U_{n-k-2i}(x) {(- U_{1}(x))}^k \right).
\end{gather*}

After changing the order of summation  \cite{BI8}

\begin{gather*}
\sum_{k=0}^n \sum_{i=0}^{\left[\frac{n-k}{2} \right]} a_{k,i} x_{n-k-2i} =\sum_{k=0}^n \left( \sum_{i=0}^{\left[\frac{n-k}{2} \right]} a_{n-(k+2i),i} \right) x_k,
\end{gather*}

and using the symmetry of binomial coefficients, we rewrite the identity $(i)$  as following 
\begin{gather*}\label{5}
\sum_{k=0}^{n}  \left(  \sum_{i=0}^{\left[\frac{n-k}{2} \right]}\frac{{{(- 1)}^{n-2i-k}(k{+}1) }}{ i{+}k{+}1 } \binom{n-i}{k+i} {n{-}k{-}i{-}1 \choose i}   {U_{1}(x)}^{n-2i-k} \right)U_{k}(x).
\end{gather*}

 $(iii)$ 
 Since the Chebyshev polynomials of the second kind are a special case of the Jacobi polynomials \cite{BI4} $P_n^{(\alpha,\beta)}$ with $\alpha=\beta=1/2$,
$$
U_n(x)=\frac{P^{\frac{1}{2},\frac{1}{2}}_n(x)}{P^{\left(\frac{1}{2},\frac{1}{2}\right)}_n(1)}=\frac{(n+1)   }{{n+\frac{1}{2} \choose n} }P^{\left(\frac{1}{2},\frac{1}{2}\right)}_n(x),
$$

we obtain the $k$-th Chebyshev polynomial derivative in the terms of the Jacobi  polynomials

\begin{gather*}
\frac{d^k}{dx^k} U_n(x)= \frac{(n+k+1)!}{2^k(n+1)!}\frac{(n+1)   }{{n+\frac{1}{2} \choose n} }
{P^{\left(k+\frac{1}{2},k+\frac{1}{2}\right)}_{n-k}(x)}=\\
=\frac{(n+k+1)!}{2^k n!{ n+\frac{1}{2} \choose n} }
{P^{\left(k+\frac{1}{2},k+\frac{1}{2}\right)}_{n-k}(x)}.
\end{gather*}

Thus, with $U_1(x)=2x,$ we obtain the identity for the special case of the Jacobi polynomials 
\begin{gather*}
 (n+1){P^{\left(\frac{1}{2},\frac{1}{2}\right)}_{n}(x)}+\sum_{k=1}^{n} \frac {(-x)^k (n+k+1)!}{ k! 2^k  n!}   {P^{\left(k+\frac{1}{2},k+\frac{1}{2}\right)}_{n-k}(x)}  =  
 {n+\frac{1}{2} \choose n}  \cos \frac{\pi n}{2}.
\end{gather*}
\end{proof}

If we rewrite the $(ii)$ identity in the terms of hypergeometric function, we obtain the following

\begin{theorem}
\begin{multline*}
\sum_{k=0}^n {n\choose k}(-2x)^{n-k}\,{}_4 F_3
\begin{bmatrix}
\begin{matrix}
&\dfrac{{-}n{+}k}{2},\dfrac{-n+k}{2}{+}1,\dfrac{{-}n{+}k}{2}{+}\dfrac 12,\dfrac{-n{+}k}{2}{+}\dfrac 12\:\\
&{-}n,k{+}2,{-}n{+}1{+}k
\end{matrix} \:
\bigg{\vert \frac{4}{{x}^{2}}}
\end{bmatrix} U_k(x)=\\
=
\cos \frac{\pi n}{2}.
\end{multline*}
\end{theorem}
\begin{proof}

Denoting  the expression under the summation   sign  by $a_i$ (with  $U_1(x)=2x$)  
$$
a_i= \frac{{{(- 1)}^{n-2i-k}(k{+}1) }}{ i{+}k{+}1 } \binom{n-i}{k+i} {n{-}k{-}i{-}1 \choose i} (2x)^{n-(k+2i)},
$$
and finding the ratio of $a_{i+1}$ and $a_i,$ after simplifying, we  obtain
\begin{gather*}
\frac{a_{i+1}}{a_{i}}={\frac { \left( k+2\,i+1-n \right) ^{2} \left( -n+k+2\,i \right) 
 \left( -n+k+2\,i+2 \right) }{  \left( -n+i \right)  \left( k+i+2
 \right)  \left( -n+k+i+1 \right)  \left( i+1 \right) 4{x}^{2}}},
\end{gather*}
with
$$
a_0={n\choose k}(-2x)^{n-k}.
$$

Using the corresponding property of the generalized hypergeometric series \cite{BI6}, we have
\begin{gather*}
\sum_{i=0}^{\infty}\frac{{{(- 1)}^{n-2i-k}(k{+}1) }}{ i{+}k{+}1 } \binom{n-i}{k+i} {n{-}k{-}i{-}1 \choose i} (2x)^{n-(k+2i)}=\\=
 {n\choose k}(-2x)^{n-k}{}_4 F_3
\begin{bmatrix}
\begin{matrix}
&\dfrac{{-}n{+}k}{2},\dfrac{-n+k}{2}{+}1,\dfrac{{-}n{+}k}{2}{+}\dfrac 12,\dfrac{-n{+}k}{2}{+}\dfrac 12\:\\
&{-}n,k{+}2,{-}n{+}1{+}k
\end{matrix} \:
\bigg{\vert \frac{4}{{x}^{2}}}
\end{bmatrix}
=\\= \sum_{i=0}^{\left[\frac{n-k}{2} \right]} \frac{{{(- 1)}^{n-2i-k}(k{+}1) }}{ i{+}k{+}1 } \binom{n-i}{k+i} {n{-}k{-}i{-}1 \choose i} (2x)^{n-(k+2i)}.
\end{gather*}
\end{proof}


\end{document}